\documentclass[a4paper, twoside]{scrartcl}
\usepackage[utf8]{inputenc}
\usepackage{amsfonts,amssymb,amsmath,amsthm,amstext,amssymb,mathtools}
\usepackage{subcaption,float,color,xcolor,graphicx,enumitem}
\usepackage{dsfont}  
\usepackage[normalem]{ulem} 
\usepackage[authoryear,sort,round]{natbib}  
\usepackage{hyperref,nicefrac}
\usepackage[noabbrev,capitalise,nosort,nameinlink]{cleveref}

\crefname{assumption}{Assumption}{Assumptions}

\addtokomafont{disposition}{\sffamily}

\usepackage{csquotes}

\DeclareMathOperator*{\dom}{\mathcal{D}}

\DeclareMathOperator*{\graph}{\mathcal{G}}

\DeclareMathOperator*{\Blim}{\textup{B-lim}}

\newcommand*{\bK}{\mathbb{K}}
\newcommand*{\cF}{\mathcal{F}}
\newcommand*{\cP}{\mathcal{P}}
\newcommand*{\cQ}{\mathcal{Q}}
\newcommand*{\Complex}{\mathbb{C}}
\newcommand*{\defeq}{\coloneqq}
\newcommand*{\defterm}{\textbf}

\renewcommand*{\epsilon}{\varepsilon}
\newcommand*{\Naturals}{\mathbb{N}}
\newcommand*{\one}{\mathds{1}}

\newcommand*{\rd}{\mathrm{d}}
\newcommand*{\Reals}{\mathbb{R}}

\renewcommand*{\leq}{\leqslant}

\renewcommand*{\mathbf}{\boldsymbol}

\newcommand{\todo}[1]{\bgroup\color{red}TODO:~#1\egroup}

\theoremstyle{plain}
\newtheorem{theorem}{\sffamily Theorem}[section]

\newtheorem{lemma}[theorem]{\sffamily Lemma}
\newtheorem{corollary}[theorem]{\sffamily Corollary}
\theoremstyle{definition}

\newtheorem{remark}[theorem]{\sffamily Remark}

\newcommand{\set}[2]{\{ #1 \mid #2 \}}

\numberwithin{equation}{section}
\numberwithin{figure}{section}
\numberwithin{table}{section}

\usepackage{acro}
\DeclareAcronym{LCS}{short=LCS, long=locally convex space}

\usepackage{etex}
\usepackage[a4paper, top=88pt, bottom=88pt, left=72pt, right=72pt, headsep=16pt, footskip=28pt]{geometry}
\usepackage{hyperref}
\hypersetup{
	breaklinks=true,
	colorlinks=true,
	linkcolor=blue,
	citecolor=blue,
	urlcolor=blue,
}
\newcommand*{\arXiv}[1]{\bgroup\color{blue}\href{https://arxiv.org/abs/#1}{arXiv:#1}\egroup}
\newcommand*{\doi}[1]{\bgroup\color{blue}\href{https://doi.org/#1}{doi:#1}\egroup}
\newcommand*{\email}[1]{\bgroup\color{blue}\href{mailto:#1}{#1}\egroup}
\renewcommand*{\url}[1]{\bgroup\color{blue}\href{#1}{#1}\egroup}
\usepackage{enumitem, moreenum}
\setlist[enumerate]{nosep}
\setlist[itemize]{nosep}
\usepackage{mleftright} \mleftright

\usepackage{xpatch}
\newcommand{\proofheadfont}{\bfseries\sffamily}
\xpatchcmd{\proof}{\itshape}{\proofheadfont}{}{}

\usepackage[labelfont={sf,bf}]{caption}
\usepackage{scrlayer-scrpage, xhfill}
\automark[section]{section}
\setkomafont{pageheadfoot}{\normalcolor\sffamily}
\setkomafont{pagenumber}{\normalfont\normalsize\sffamily}
\clearpairofpagestyles
\let\oldtitle\title
\renewcommand{\title}[1]{\oldtitle{#1}\newcommand{\theshorttitle}{#1}}
\newcommand{\shorttitle}[1]{\renewcommand{\theshorttitle}{#1}}
\let\oldauthor\author
\renewcommand{\author}[1]{\oldauthor{#1}\newcommand{\theshortauthor}{#1}}
\newcommand{\shortauthor}[1]{\renewcommand{\theshortauthor}{#1}}
\cohead{\xrfill[0.525ex]{0.6pt}~\theshorttitle~\xrfill[0.525ex]{0.6pt}}
\cehead{\xrfill[0.525ex]{0.6pt}~\theshortauthor~\xrfill[0.525ex]{0.6pt}}
\newcommand{\theabstract}[1]{\par\bgroup\noindent\textbf{\textsf{Abstract.}} #1\egroup}
\newcommand{\thekeywords}[1]{\par\smallskip\bgroup\noindent\textbf{\textsf{Keywords.}}\newcommand{\and}{ $\bullet$ } #1\egroup}
\newcommand{\themsc}[1]{\par\smallskip\bgroup\noindent\textbf{\textsf{2020 Mathematics Subject Classification.}}\newcommand{\and}{ $\bullet$ } #1\egroup}

\newcommand*{\affilref}[1]{\ref{affiliation#1}}
\newcommand*{\affiliation}[3]{
	\footnotetext[#1]{\label{affiliation#2}#3}
}

\title{Hille's theorem for Bochner integrals of functions with values in locally convex spaces}
\shorttitle{Hille's theorem in locally convex spaces}
\author{
	T.~J.~Sullivan\textsuperscript{\affilref{Warwick}}
}
\shortauthor{T.~J.~Sullivan}
\date{\today}

\begin{document}
\maketitle
\affiliation{1}{Warwick}{Mathematics Institute and School of Engineering, University of Warwick, Coventry, CV4 7AL, United Kingdom\newline (\email{t.j.sullivan@warwick.ac.uk})}

\begin{abstract}\small
	\theabstract{Hille's theorem is a powerful classical result in vector measure theory.
It asserts that the application of a closed, unbounded linear operator commutes with strong/Bochner integration of functions taking values in a Banach space.
This note shows that Hille's theorem also holds in the setting of complete locally convex spaces.
}
	\thekeywords{Bochner integral%
\and%
closed operator%
\and%
Hille's theorem%
\and%
locally convex space%
\and%
strong integral%
\and%
unbounded operator%
}
	\themsc{28B05
\and%
28C20
\and%
46G10
}
\end{abstract}

\section{Introduction}

The strong or Bochner integral of a function taking values in a Banach space $V$ is a long-established mathematical concept \citep{Bochner1935}, as are various weaker integrals such as the $V''$-valued Dunford integral \citep{Dunford1937,Dunford1938} and the $V$-valued weak and Gel$'$fand--Pettis integrals (\citealp{Gelfand1936}; \citealp{Pettis1938}).
The extension of the weaker integrals to integrands taking values in more general topological vector spaces is relatively straightforward, since those integrals are characterised though the application of continuous linear functionals.
The strong integral can also be generalised to \acp{LCS} $V$, as has been done by e.g.\ \citet{Blondia1981}.
Integrability by seminorm (\citealp{GarnirDeWildeSchmets1972}; \citealp{Blondia1981}) is intermediate between the strong and weak integrals, although it agrees with the strong/Bochner integral in the Banach case.

One of the key features of the strong/Bochner integral, in the Banach case, is Hille's theorem \citep[p.83]{HillePhillips1957}, which asserts that integration commutes not merely with bounded/continuous linear operators, but even with \emph{closed} unbounded operators.
This non-trivial fact has no analogue for the weak integral --- q.v.\ standard counterexamples to the principle of differentiation under the integral sign when the domination hypothesis fails \citep{Talvila2001b} --- and is of vital importance in, e.g., understanding the mean and covariance structure of linear images of stochastic processes under differential operators.
This note proves the following Hille-type theorem for the Bochner integral of a function taking values in an LCS, which does not appear to have been explicitly noted in the literature hitherto.

\begin{theorem}[Hille's theorem for the Bochner integral in an \ac{LCS}]
	\label{thm:Hille}
	Let $(\Omega, \cF, \mu)$ be a finite measure space and let $V$ and $W$ be complete \acp{LCS} over the same field $\bK = \Reals$ or $\Complex$.
	Let $T \colon \dom(T) \subseteq V \to W$ be a closed operator defined on a linear subspace $\dom(T)$ of $V$ and suppose that both $u \colon \Omega \to V$ and $T u \colon \Omega \to W$ are Bochner integrable with respect to $\mu$.
	Then, for every $A \in \cF$,
	\begin{align}
		\label{eq:Hille}
		\tag{\textup{H}}
		\int_{A} u \, \rd \mu \in \dom(T)
		\quad
		\text{and}
		\quad
		\int_{A} T u \, \rd \mu = T \int_{A} u \, \rd \mu .
	\end{align}
\end{theorem}

The terms and concepts used in the statement of \Cref{thm:Hille} are given in \Cref{sec:background}, and a pedagogically-paced proof of \Cref{thm:Hille} is given in \Cref{sec:proof}.
\Cref{sec:by_seminorm} discusses how the proof of Hille's theorem fails for the next-strongest notion of integral, namely integration by seminorm.

\section{Background and notation}
\label{sec:background}

\paragraph{Locally convex spaces.}
Throughout, $V$ and $W$ will denote \aclp{LCS} over the same field $\bK = \Reals$ or $\mathbb{C}$;
let $\cP$ and $\cQ$ denote collections of continuous seminorms $p \colon V \to \Reals$ and $q \colon W \to \Reals$ that generate the topologies of $V$ and $W$ respectively.
The topological dual space of $V$, comprised of all continuous linear functionals from $V$ into $\bK$, is denoted by $V'$.
As is well known, $V$ is first countable if and only if $\cP$ can be taken to be a finite or countable collection $\{ p_{n} \}_{n \in \mathcal{N}}$, $\mathcal{N} \subseteq \Naturals$, and it is precisely in this situation that $V$ is pseudometrisable, with
\begin{align}
	\label{eq:pseudometric}
	d(v, v') \defeq \sum_{n \in \mathcal{N}} 2^{-n} \frac{p_{n}(v - v')}{1 + p_{n}(v - v')}
\end{align}
being a translation-invariant pseudometric\footnote{I.e.\ $d \colon V \times V \to [0, \infty)$ satisfies all the requirements to be a metric with the exception that $d(v, v') = 0$ need not imply that $v = v'$.} that generates the topology of $V$.
The space $V$ is Hausdorff precisely when $\cP$ is a separating family of seminorms, i.e.\ when, for each $v \in V \setminus \{ 0 \}$, there is some $p \in \cP$ with $p(v) > 0$;
thus, $V$ is metrisable if and only if it admits a countable and separating family $\cP$.
If $V$ is completely metrisable, then it is called a \defterm{Fr\'echet space}, and its topology is induced by the translation-invariant metric \eqref{eq:pseudometric}.

For simplicity, it will be assumed henceforth that $V$ and $W$ are both Hausdorff;
if this assumption fails, then certain statements such as \eqref{eq:Hille} about quantities being equal would have to be amended to state that they are topologically indistinguishable.

\paragraph{Closed operators.}
A linear operator $T \colon \dom(T) \subseteq V \to W$ between $V$ and $W$ is called a \defterm{closed operator} if its \defterm{graph} $\graph(T) \defeq \set{ (v, T v) \in V \oplus W }{ v \in \dom(T) }$ is a closed subspace of $V \oplus W$ when it is equipped with the family of seminorms
\begin{align}
	\label{eq:graph_seminorms_V_oplus_W}
	(v, w) \mapsto p(v) + q(w) \quad
	\text{for $p \in \cP$, $q \in \cQ$.}
\end{align}
Explicitly, $T$ is closed if, whenever $(v_{j})_{j}$ is a net in $\dom(T)$ converging in $V$ to some $v$, such that $(T v_{j})_{j}$ converges in $W$ to some $w$, it follows that $v \in \dom(T)$ and $w = T v$.

\paragraph{Measurable functions.}
Let $(\Omega, \cF, \mu)$ be a measure space with $0 < \mu(\Omega) < \infty$ and let $V$ be a topological space equipped with its Borel $\sigma$-algebra.
A function $u \colon \Omega \to V$ is called \defterm{simple}\footnote{\label{fnote:simples}Note that the terminology and definitions used for simple and $\sigma$-simple functions ensure that they are Borel measurable, and indeed strongly measurable, which might be in doubt if the alternative terminology of ``finitely-valued'' and ``countably-valued'' functions were used.} if there exists a partition of $\Omega$ into finitely many sets $E_{1}, \dots, E_{K} \in \cF$ such that $u$ is constant on each $E_{k}$, and it is called \defterm{$\sigma$-simple}\footref{fnote:simples} if there is a countable partition with this same property.

A function $u \colon \Omega \to V$ is \defterm{$\mu$-essentially separably valued} if there exists $Z \in \cF$ with $\mu(Z) = 0$ and a countable set $C \subseteq V$ such that $f(\Omega \setminus Z) \subseteq \overline{C}$.

A function $u \colon \Omega \to V$ is called \defterm{Borel measurable} if, for each open set $A \subseteq V$, $u^{-1}(A) \in \cF$.
A function $u \colon \Omega \to V$ is called \defterm{strongly $\mu$-measurable} if it is a $\mu$-a.e.\ pointwise limit of a sequence of simple functions.
When $V$ is a topological vector space, $u$ is called \defterm{weakly $\mu$-measurable} if, for all $\ell \in V'$, $\ell u \colon \Omega \to \bK$ is strongly measurable.\footnote{From this point on, the prefix ``$\mu$-'' will be omitted from terms like measurability, essential separability, etc.}
Strong measurability implies Borel measurability, which in turn implies weak measurability.
As is well known, for a Banach space $V$, $u$ is strongly measurable if and only if it is weakly measurable and essentially separably valued (Pettis' measurability theorem; \citealp[Theorem~II.1.2]{DiestelUhl1977}).

When $V$ is a topological vector space, or indeed an \ac{LCS}, a simple function $u \colon \Omega \to V$ can be expressed as a linear combination of the form $u = \sum_{k = 1}^{K} v_{k} \one_{E_{k}}$ for some $K \in \Naturals$, $v_{1}, \dots, v_{K} \in V$, and pairwise disjoint sets $E_{1}, \dots, E_{K} \in \cF$, where $\one_{E} \colon \Omega \to \Reals$ denotes the indicator function
\begin{align*}
	\one_{E} (\omega) \defeq
	\begin{cases}
		1, & \text{if $\omega \in E$,} \\
		0, & \text{if $\omega \notin E$;}
	\end{cases}
\end{align*}
similarly, a $\sigma$-simple function $u \colon \Omega \to V$ can be expressed in the form $u = \sum_{k \in \Naturals} v_{k} \one_{E_{k}}$ for some $v_{1}, v_{2}, \dots \in V$, and pairwise disjoint $E_{1}, E_{2}, \dots \in \cF$.

\paragraph{Bochner approximation and integration.}
The \defterm{integral} of a $V$-valued simple function with respect to $\mu$ over $A \in \cF$ is unambiguously defined by
\begin{align}
	\label{eq:simple_integral}
	\int_{A} \left( \sum_{k = 1}^{K} v_{k} \one_{E_{k}} \right) \, \rd \mu \defeq \sum_{k = 1}^{K} v_{k} \mu ( A \cap E_{k} ) \in V .
\end{align}
Of course, the entire point of integration theory is to integrate non-simple functions.

A net or sequence of measurable functions $u_{j} \colon \Omega \to V$ is said to converge to $u \colon \Omega \to V$ in the \defterm{Bochner sense}, denoted $\Blim_{j} u_{j} = u$, if, for all $p \in \cP$,
\begin{align*}
	p(u_{j} - u) \to 0 \text{ a.e.\ and } \int_{\Omega} p(u_{j} - u) \, \rd \mu \to 0 .
\end{align*}
Note that uniform convergence on a full-measure subset of $\Omega$ implies Bochner convergence, but that one should be careful about referring to a \emph{topology} of Bochner convergence since, as is well known, convergence a.e.\ is not topological \citep{Ordman1966}.

A function $u \colon \Omega \to V$ is called \defterm{Bochner approximable} if there exists a sequence\footnote{Note that \citet[Definition~1.11]{BeckmannDeitmar2015} define a function to be Bochner approximable if it is a Bochner limit of a \emph{net} of simple functions, but assume a priori that the limit function is Borel measurable.
In contrast, a Bochner limit of a \emph{sequence} of simple functions is necessarily strongly measurable.} $(s_{n})_{n \in \Naturals}$ of simple functions $s_{n} \colon \Omega \to V$ such that $\Blim_{n \to \infty} s_{n} = u$.
In this situation, $u$ is called \defterm{Bochner integrable} if, for each $A \in \cF$, $( \int_{A} s_{n} \, \rd \mu )_{n}$ converges in $V$, and the \defterm{Bochner integral} of $u$ with respect to $\mu$ over $A$ is defined by
\begin{align*}
	\int_{A} u \, \rd \mu \defeq \lim_{n \to \infty} \int_{A} s_{n} \, \rd \mu .
\end{align*}

The usual arguments show that, if $\bigl( \int_{A} s_{n} \, \rd \mu \bigr)_{n}$ converges for one sequence $(s_{n})_{n}$ of simple functions that converges in the Bochner sense, then it converges for every such sequence, yielding the same limiting values for $\int_{A} u \, \rd \mu$ in each case.
It is also straightforward to show that, if $(s_{n})_{n}$ is a Bochner-convergent sequence of simple functions, then $\bigl( \int_{A} s_{n} \, \rd \mu \bigr)_{n}$ is Cauchy.
Thus, in the case of a complete space $V$, every Bochner approximable function is Bochner integrable.

Whenever $u \colon \Omega \to V$ is Bochner integrable, the integral triangle inequality holds for each continuous seminorm $p$ on $V$:
\begin{align}
	\label{eq:integral_triangle_ineq}
	p \left( \int_{A} u \, \rd \mu \right) \leq \int_{A} p(u) \, \rd \mu < \infty .
\end{align}
The property that the real-valued Lebesgue integral $\int_{\Omega} p(u) \, \rd \mu$ exists and is finite for each continuous seminorm $p$ is referred to as $u$ being \defterm{integrally bounded}.

\section{Hille's theorem for the Bochner integral in an LCS}
\label{sec:proof}

It is perhaps worth noting that Hille's theorem is almost trivial to prove for any simple integrand.
Thus, there would seem to be an obvious approach to proving Hille's theorem for Bochner-integrable functions:
one would take a sequence $(s_{n})_{n}$ of simple functions that approximate $u$ in the Bochner sense, reason that $(T s_{n})_{n}$ is a sequence of simple functions that approximates $T u$, and take the limit as $n \to \infty$ of the corresponding integrals.
Unfortunately, almost everything in this approach is wrong:
first, there is no guarantee that $s_{n}$ takes values only in $\dom(T)$;
second, even if this obstacle is resolved, there is no guarantee that the simple functions $T s_{n}$ converge to anything, let alone to $T u$, since $T$ is not continuous.

The correct way to proceed is to approximate $u$ and $T u$ independently, at least at first, and moreover to work not with a sequence of simple functions that converges a.e.\ to the approximand, but a sequence/net of $\sigma$-simple functions that converges uniformly on a set of full measure.
This trick is closely related to Egorov's theorem, which has analogues for metrisable \acp{LCS} and LF spaces \citep[Theorems~4 and 7]{deMaria1984}.
However, the following lemma establishes the required result in a different way, exploiting the relationship between essential separability and approximation using $\sigma$-simple functions when the approximand is known a priori to be Borel measurable.

\begin{lemma}[Essentially uniform $\sigma$-simple approximation of essentially separable functions]
	\label{lem:essential_separability_implies_uniform_countable_approx}
	Let $(\Omega, \cF, \mu)$ be a finite measure space, let $V$ be a uniform space, and let $u \colon \Omega \to V$ be Borel measurable and essentially separably valued.
	Let $\Lambda$ be a directed index set for a fundamental system of entourages for $V$, and hence for a neighbourhood basis of any point of $V$.
	Then there exists $Z \in \cF$ with $\mu(Z) = 0$ and a net $(\phi_{\lambda})_{\lambda \in \Lambda}$ of $\sigma$-simple functions $\phi_{\lambda} \colon \Omega \to V$ such that $u = \lim_{\lambda} \phi_{\lambda}$ uniformly on $\Omega \setminus Z$.
\end{lemma}

\begin{remark}
	Readers unfamiliar with uniform spaces may wish to follow the proof below thinking of the case that $V$ is a metric space, $\Lambda = \Naturals$, and $N_{\lambda}[v]$ is the metric ball centred on $v$ with radius $\nicefrac{1}{\lambda}$.
	In this situation, the approximating net can be taken to be a sequence and the proof concludes with the estimate that $\sup_{\Omega \setminus Z} d(u, \phi_{\lambda}) \leq \nicefrac{1}{\lambda}$.
\end{remark}

\begin{proof}
	Let $\{ U_{\lambda} \}_{\lambda \in \Lambda}$ be a fundamental system of entourages for the uniform space $V$, so that the neighbourhoods $N_{\lambda}[v] \defeq \set{ v' \in V }{ (v, v') \in U_{\lambda} }$, $\lambda \in \Lambda$, form a neighbourhood basis for each point $v \in V$, directed by inclusion in the usual way.

	Let $Z \in \cF$ with $\mu(Z) = 0$ and $C = \{ c_{k} \}_{k \in \Naturals} \subseteq V$ be such that $u(\Omega \setminus Z) \subseteq \overline{C}$.
	For each $\lambda \in \Lambda$ and $k \in \Naturals$, let $G_{\lambda, k} \defeq N_{\lambda}[c_{k}] \setminus \bigcup_{\ell < k} N_{\lambda}[c_{\ell}]$, so that, for each $\lambda \in \Lambda$, we have a measurable, countable, pairwise disjoint covering
	\[
		\biguplus_{k \in \Naturals} G_{\lambda, k} \supseteq \overline{C} \supseteq u(\Omega \setminus Z) .
	\]
	Define $\phi_{\lambda} \colon \Omega \to V$ to take the value $c_{k}$ on the set $u^{-1}(G_{\lambda, k})$, i.e., in the case that $V$ is a topological vector space,
	\[
		\phi_{\lambda} (\omega) \defeq \sum_{k \in \Naturals} c_{k} \one_{u^{-1}(G_{\lambda, k})} (\omega) .
	\]
	(Heuristically, to define $\phi_{\lambda}(\omega)$, one considers all the neighbourhoods $N_{\lambda}[c_{k}]$ that contain $u(\omega)$, and then ``rounds'' $u(\omega)$ to the $c_{k}$ with least $k$.)

	Clearly, each $\phi_{\lambda}$ assumes at most countably many distinct values;
	each $\phi_{\lambda}$ is measurable since the sets $G_{\lambda, k}$ are measurable and hence, by the Borel measurability of $u$, each $u^{-1}(G_{\lambda, k}) \in \cF$.
	By construction, for all $\omega \in \Omega \setminus Z$, $\phi_{\lambda}(\omega) \in N_{\lambda}[u(\omega)]$, and so $\lim_{\lambda} \phi_{\lambda} = u$ uniformly on $\Omega \setminus Z$, as required.
\end{proof}

There is also a straightforward result in the converse direction:

\begin{lemma}[Essential separability of pointwise sequential limits of $\sigma$-simple functions]
	\label{lem:countable_approx_implies_essential_separability}
	Let $(\Omega, \cF, \mu)$ be a finite measure space and let $V$ be a topological space.
	If $u \colon \Omega \to V$ satisfies $u = \lim_{n \to \infty} \phi_{n}$ a.e.\ for a sequence $(\phi_{n})_{n \in \Naturals}$ of $\sigma$-simple functions $\phi_{n} \colon \Omega \to V$, then $u$ is essentially separably valued.
\end{lemma}

\begin{proof}
	Let $Z \in \cF$ with $\mu(Z) = 0$ be such that $u(\omega) = \lim_{n \to \infty} \phi_{n}(\omega)$ for all $\omega \in \Omega \setminus Z$.
	Let $C \defeq \set{ \phi_{n}(\omega) }{ n \in \Naturals, \omega \in \Omega }$;
	the hypotheses on $\phi_{n}$ and the Axiom of Countable Choice imply that $C$ is a countable set.
	Also, for every $\omega \in \Omega \setminus Z$, $u(\omega) = \lim_{n \to \infty} \phi_{n}(\omega) \in \overline{C}$.
	Thus, $u$ is essentially separably valued.
\end{proof}

Combining \Cref{lem:essential_separability_implies_uniform_countable_approx,lem:countable_approx_implies_essential_separability} immediately yields the following:

\begin{corollary}[Essentially uniform $\sigma$-simple approximation of strongly measurable functions]
	\label{cor:ess_unif_sigma_simple_approx_of_strongly_measurable}
	Let $u \colon \Omega \to V$ be a strongly measurable function with values in an \ac{LCS} $V$ with a neighbourhood basis of the origin indexed by $\Lambda$.
	Then there exists $Z \in \cF$ with $\mu(Z) = 0$ and a net $(\phi_{\lambda})_{\lambda \in \Lambda}$ of $\sigma$-simple functions $\phi_{\lambda} \colon \Omega \to V$ such that, for every continuous seminorm $p$ on $V$,
	\[
		\lim_{\lambda} \sup_{\omega \in \Omega \setminus Z} p ( u(\omega) - \phi_{\lambda}(\omega) ) = 0 .
	\]
\end{corollary}

With the above ingredients in place, the proof of Hille's theorem proceeds much as in the Banach space case, as done by e.g.\ \citet[Theorem~II.2.6]{DiestelUhl1977}.

\begin{proof}[Proof of \Cref{thm:Hille}]
	\textit{\textsf{Step 1: Approximation of $u$ and $T u$.}}
	Since $u$ and $T u$ are Bochner integrable, there are sequences $(s_{n})_{n \in \Naturals}$ and $(t_{n})_{n \in \Naturals}$ of simple functions $s_{n} \colon \Omega \to V$ and $t_{n} \colon \Omega \to W$ such that $\Blim_{n} s_{n} = u$ and $\Blim_{n} t_{n} = T u$.
	For each $n \in \Naturals$, $(s_{n}, t_{n}) \colon \Omega \to V \oplus W$ is again a simple function, and it is easy to verify that $\Blim_{n} (s_{n}, t_{n}) = (u, T u)$ with respect to the seminorms \eqref{eq:graph_seminorms_V_oplus_W}.
	Hence, $(u, T u)$ is Bochner approximable.
	However, we still face the issue that, although $(u, T u)$ takes values in $\graph(T)$, we do not yet have a $\graph(T)$-valued approximating sequence;
	there is no guarantee that $s_{n}$ takes values in $\dom(T)$, nor that $T s_{n} = t_{n}$.

	Let $\Lambda$ be a directed set indexing a neighbourhood basis of the origin in $V \oplus W$.
	By \Cref{cor:ess_unif_sigma_simple_approx_of_strongly_measurable}, there exists $Z_{1} \in \cF$ with $\mu(Z_{1}) = 0$ and a net $(\phi_{\lambda}, \psi_{\lambda})_{\lambda \in \Lambda}$ of $\sigma$-simple functions $(\phi_{\lambda}, \psi_{\lambda}) \colon \Omega \to V \oplus W$ such that $\lim_{\lambda} (\phi_{\lambda}, \psi_{\lambda}) = (u, T u)$ uniformly on $\Omega \setminus Z_{1}$.
	There also exists $Z_{2} \in \cF$ with $\mu(Z_{2}) = 0$ such that $u(\omega) \in \dom(T)$ for all $\omega \in \Omega \setminus Z_{2}$.
	Let $Z \defeq Z_{1} \cup Z_{2}$.
	Thus, for each $\lambda \in \Lambda$, there is a countable partition\footnote{Of course, the partition into non-empty level sets of $(\phi_{\lambda}, \psi_{\lambda})$ could actually be a finite one, but for the sake of notational simplicity this case will be subsumed into the countably infinite case without further comment.} $\Omega \setminus Z = \biguplus_{k \in \Naturals} E_{\lambda, k}$ into non-empty $\cF$-measurable sets $E_{\lambda, k}$ such that $(\phi_{\lambda}, \psi_{\lambda})$ is constant on each $E_{\lambda, k}$.
	Let $\omega_{\lambda, k} \in E_{\lambda, k}$ be arbitrary and define a $\sigma$-simple function $u_{\lambda} \colon \Omega \to \dom(T) \subseteq V$ by
	\begin{align}
		\label{eq:Hille_Bochner_approx_un}
		u_{\lambda} \defeq \sum_{k \in \Naturals} u(\omega_{\lambda, k}) \one_{E_{\lambda, k}} .
	\end{align}
	Note that, by construction, $u_{\lambda} = 0$ on $Z$.
	Note also that $T u_{\lambda} \defeq \sum_{k \in \Naturals} T u(\omega_{\lambda, k}) \one_{E_{\lambda, k}}$ is well defined and that $u_{\lambda} \to u$ and $T u_{\lambda} \to T u$ essentially uniformly:
	for all $p \in \cP$ and $q \in \cQ$,
	\begin{align}
		\label{eq:un_to_u_essunif}
		\gamma_{\lambda, p} \defeq \sup_{\Omega \setminus Z} p ( u - u_{\lambda} ) \leq 2 \sup_{\Omega \setminus Z} p ( u - \phi_{\lambda} ) \to 0 , \\
		\label{eq:Tun_to_Tu_essunif}
		\Gamma_{\lambda, q} \defeq \sup_{\Omega \setminus Z} q ( T u - T u_{\lambda} ) \leq 2 \sup_{\Omega \setminus Z} q ( T u - \psi_{\lambda} ) \to 0 .
	\end{align}

	\textit{\textsf{Step 2: Integration of approximants $u_{\lambda}$ and $T u_{\lambda}$.}}
	Fix $\lambda \in \Lambda$.
	By construction, $u_{\lambda}$ is essentially separable and strongly measurable:
	the simple partial sums $u_{\lambda, K} \defeq \sum_{k = 1}^{K} u(\omega_{\lambda, k}) \one_{E_{\lambda, k}}$ of the series \eqref{eq:Hille_Bochner_approx_un} converge pointwise to $u_{\lambda}$ as $K \to \infty$.
	Consider any $p \in \cP$.
	Integrating the estimate
	\begin{align*}
		p ( u(\omega_{\lambda, k}) )
		\leq p ( u(\omega_{\lambda, k}) - u(\omega) ) + p ( u(\omega) )
		\leq \gamma_{\lambda, p} + p ( u(\omega) ) \text{ for $\omega \in E_{\lambda, k}$}
	\end{align*}
	over $E_{\lambda, k}$ and then summing over $k$ shows that
	\[
		\sum_{k \in \Naturals} p ( u(\omega_{\lambda, k}) ) \mu ( E_{\lambda, k} ) \leq \gamma_{\lambda, p} \mu(\Omega) + \int_{\Omega} p(u) \, \rd \mu < \infty .
	\]
	Lebesgue's dominated convergence theorem then yields that
	\[
		\int_{\Omega} p ( u_{\lambda} - u_{\lambda, K} ) \, \rd \mu
		=
		\int_{\Omega} p \left(\sum_{k > K} u(\omega_{\lambda, k}) \one_{E_{\lambda, k}} \right) \, \rd \mu
		\leq
		\sum_{k > K} p ( u(\omega_{\lambda, k}) ) \mu ( E_{\lambda, k} ) \to 0 \text{ as $K \to \infty$.}
	\]
	Hence, $\Blim_{K \to \infty} u_{\lambda, K} = u_{\lambda}$, which shows that $u_{\lambda}$ is Bochner approximable;
	by the completeness of $V$, $u_{\lambda}$ is Bochner integrable.
	A similar argument shows that $T u_{\lambda}$ is Bochner integrable and, indeed, that $(u_{\lambda}, T u_{\lambda})$ is Bochner integrable.

	We now show that Hille's theorem holds for each of these approximants.
	Consider the partial sums $u_{\lambda, K}$, each with Bochner integral with respect to $\mu$ over $A \in \cF$ given by \eqref{eq:simple_integral}, i.e.
	\[
		\int_{A} u_{\lambda, K} \, \rd \mu = \sum_{k = 1}^{K} u(\omega_{\lambda, k}) \mu ( A \cap E_{\lambda, k} ) \in \dom(T) .
	\]
	By the previous paragraph, $\bigl( \int_{A} u_{\lambda, K} \, \rd \mu \bigr)_{K \in \Naturals}$ is a sequence in $\dom(T)$ that converges in $V$ to $\int_{A} u_{\lambda} \, \rd \mu$ as $K \to \infty$.
	Also, for each $K \in \Naturals$,
	\[
		T \left( \int_{A} u_{\lambda, K} \, \rd \mu \right)
		= T \sum_{k = 1}^{K} u(\omega_{\lambda, k}) \mu ( A \cap E_{\lambda, k} )
		= \sum_{k = 1}^{K} T u(\omega_{\lambda, k}) \mu ( A \cap E_{\lambda, k} )
		= \int_{A} T u_{\lambda, K} \, \rd \mu ,
	\]
	and this converges in $W$ to $\int_{A} T u_{\lambda} \, \rd \mu$ as $K \to \infty$.
	Therefore, since $T$ is a closed operator,
	\begin{align}
		\label{eq:Hille_for_approx}
		\int_{A} u_{\lambda} \, \rd \mu \in \dom(T)
		\quad
		\text{and}
		\quad
		\int_{A} T u_{\lambda} \, \rd \mu = T \int_{A} u_{\lambda} \, \rd \mu .
	\end{align}

	\textit{\textsf{Step 3: Integration of $u$ and $T u$.}}
	For each $p \in \cP$, the triangle inequality \eqref{eq:integral_triangle_ineq} and the essentially uniform convergence of $u_{\lambda}$ to $u$ \eqref{eq:un_to_u_essunif} imply that
	\[
		p \left( \int_{A} u \, \rd \mu - \int_{A} u_{\lambda} \, \rd \mu \right)
		\leq
		\int_{A} p(u - u_{\lambda}) \, \rd \mu
		\leq
		\gamma_{\lambda, p} \mu(A)
		\to
		0 ;
	\]
	similarly, for each $q \in \cQ$,
	\begin{align*}
		q \left( \int_{A} T u \, \rd \mu - T \int_{A} u_{\lambda} \, \rd \mu \right)
		& =
		q \left( \int_{A} T u \, \rd \mu - \int_{A} T u_{\lambda} \, \rd \mu \right) & & \text{by \eqref{eq:Hille_for_approx}} \\
		& \leq
		\int_{A} q(T u - T u_{\lambda}) \, \rd \mu & & \text{by \eqref{eq:integral_triangle_ineq}} \\
		& \leq
		\Gamma_{\lambda, q}
		\to
		0 & & \text{by \eqref{eq:Tun_to_Tu_essunif}.}
	\end{align*}
	Thus, $\bigl( \int_{A} u_{\lambda} \, \rd \mu \bigr)_{\lambda}$ is a net in $\dom(T)$ that converges in $V$ to $\int_{A} u \, \rd \mu$, and $\bigl( T \int_{A} u_{\lambda} \, \rd \mu \bigr)_{\lambda}$ converges in $W$ to $\int_{A} T u \, \rd \mu$.
	Since $T$ is closed,
	\begin{align*}
		\int_{A} u \, \rd \mu \in \dom(T)
		\quad
		\text{and}
		\quad
		\int_{A} T u \, \rd \mu = T \int_{A} u \, \rd \mu ,
	\end{align*}
	i.e.\ \eqref{eq:Hille} holds, completing the proof.
\end{proof}

\begin{remark}
	\label{rmk:alternative}
	Some readers may prefer to view the steps in the above proof of \Cref{thm:Hille} slightly differently:

	Step~1 constructs a net $(u_{\lambda}, T u_{\lambda})_{\lambda \in \Lambda}$ of $\graph(T)$-valued $\sigma$-simple approximations to $(u, T u)$ that converge essentially uniformly to $(u, T u)$ with respect to the seminorms defined in \eqref{eq:graph_seminorms_V_oplus_W}.

	Step~2 verifies that each term in this net is Bochner integrable and shows that the sequence $\bigl( \int_{A} u_{\lambda, K} \, \rd \mu , \int_{A} T u_{\lambda, K} \, \rd \mu \bigr)_{K \in \Naturals}$ is a convergent sequence with all its terms in the closed set $\graph(T)$;
	hence its limit as $K \to \infty$, namely $\bigl( \int_{A} u_{\lambda} \, \rd \mu , \int_{A} T u_{\lambda} \, \rd \mu \bigr)$, belongs to $\graph(T)$ as well.

	Step~3 employs the same closedness argument to show that the limit with respect to $\lambda \in \Lambda$, namely $\bigl( \int_{A} u \, \rd \mu , \int_{A} T u \, \rd \mu \bigr)$, also belongs to $\graph(T)$, which is just another way of formulating the claim \eqref{eq:Hille} of Hille's theorem.
\end{remark}

\section{Integration by seminorm}
\label{sec:by_seminorm}

Integrability by seminorm appears to have been introduced by \citet{GarnirDeWildeSchmets1972} and then developed further by \citet{Blondia1981};
see \citet{Lewis2022} for a recent survey and further references.
Although essential separability by seminorm, measurability by seminorm, and integrability by seminorm agree with the usual notions of essential separability, strong measurability, and strong integrability when $V$ is a Banach space, they do differ in general, and it is instructive to briefly review the concepts and see how the proof of Hille's theorem breaks down in the ``by seminorm'' setting.

A function $u \colon \Omega \to V$ is \defterm{integrable by seminorm} if, for each $p \in \cP$, there exists a sequence $(s_{n}^{p})_{n \in \Naturals}$ of simple functions $s_{n}^{p} \colon \Omega \to V$ such that
\begin{enumerate}[label=(\alph*)]
	\item $\lim_{n \to \infty} p(s_{n}^{p} - u) = 0$ pointwise on $\Omega \setminus Z_{p}$, where $\mu(Z_{p}) = 0$, i.e.\ $u$ is \defterm{measurable by seminorm};
	\item $\lim_{n \to \infty} \int_{\Omega} p(s_{n}^{p} - u) \, \rd \mu = 0$;
\end{enumerate}
and there exists, for each $A \in \cF$, some element $\int_{A} u \, \rd \mu \in V$ such that, for all $p \in \cP$,
\[
	\lim_{n \to \infty} p \left( \int_{A} s_{n}^{p} \, \rd \mu - \int_{A} u \, \rd \mu \right) = 0 .
\]

Borel measurability always implies measurability by seminorm;
if the space $V$ is metrisable, then measurability by seminorm implies strong measurability and hence Borel measurability \citep[Proposition~2.4]{Lewis2022}.
There is also a version of Pettis' measurability theorem for the ``by seminorm'' setting:
measurability by seminorm is equivalent to weak measurability together with essential separability by seminorm\footnote{A function $u \colon \Omega \to V$ is \defterm{essentially separably valued by seminorm} if, for each $p \in \cP$, there exists $Z_{p} \in \cF$ with $\mu(Z_{p}) = 0$ and a countable set $C_{p} \subseteq V$ such that $u ( \Omega \setminus Z_{p} )$ is contained within the closure of $C_{p}$ with respect to the seminorm $p$.
Confusingly, \citet[Definition~1.4]{BeckmannDeitmar2015} refer to essential separability by seminorm as essential separability.
The two concepts are not at all equivalent:
both $\bigcup_{p \in \cP} Z_{p} \subseteq \Omega$ and $\bigcup_{p \in \cP} C_{p} \subseteq V$ may be very large sets.} \citep[Theorem~2.3]{Lewis2022}.

If $V$ is complete, then integrability by seminorm is equivalent to measurability by seminorm together with integral boundedness (\citealp[Theorem~2.10]{Blondia1981}; \citealp[Theorem~2.7]{Lewis2022}).

Crucially, when defining integrability by seminorm, both the approximating sequence of simple functions $(s_{n}^{p})_{n}$ and the exceptional $\mu$-null set $Z_{p} \in \cF$ on which $p(s_{n}^{p} - u) \not\to 0$ may both depend on $p$.
If $V$ is not first countable, then it may be the case that $\bigcup_{p \in \cP} Z_{p} \notin \cF$ or that it has positive outer measure.
There is also no guarantee that
\[
	\lim_{n \to \infty} \int_{\Omega} \widetilde{p}(s_{n}^{p} - u) \, \rd \mu = 0
\]
for distinct continuous seminorms $p, \widetilde{p} \in \cP$.
It is for this reason that the above proof of Hille's theorem for the strong/Bochner integral fails for integration by seminorm:
the analogue of Step~1 would produce a family of a approximating nets $\bigl( \phi_{\lambda}^{(p, q)}, \psi_{\lambda}^{(p, q)} \bigr)_{\lambda \in \Lambda}$, one for each $p \in \cP$ and $q \in \cQ$, each converging essentially uniformly to $(u, T u)$ only with respect to those two seminorms $p$ and $q$.
Combining all these nets to form a single approximation $u_{\lambda}$, as in \eqref{eq:Hille_Bochner_approx_un}, that provides essentially uniform approximation of $u$ in $V$ (and of $T u$ in $W$), would seem to be challenging if not impossible.
Without this, Step~3 would fail, because one would not know, e.g., that $\bigl( \int_{A} u_{\lambda} \, \rd \mu \bigr)_{\lambda}$ converges in $V$, and hence there would be no scope to appeal to the closedness of the operator $T$.

\section*{Acknowledgements}
\addcontentsline{toc}{section}{Acknowledgements}

This note was inspired by a question posed by M.\ Lange--Hegermann, for which the author expresses his thanks.
The author also thanks H.\ Lambley and an anonymous peer reviewer for helpful comments on the manuscript.

For the purpose of open access, the author has applied a Creative Commons Attribution (CC~BY) licence to any Author Accepted Manuscript version arising.
No data were created or analysed in this study.

\bibliographystyle{abbrvnat}
\bibliography{references}
\addcontentsline{toc}{section}{References}

\end{document}